\documentclass[11pt,reqno]{amsart}
\usepackage{a4wide}
\usepackage[english]{babel}
\usepackage{hyperref}
\usepackage{amsfonts,amsmath,amssymb,amsthm}
\usepackage[foot]{amsaddr}
\usepackage{latexsym}
\usepackage{layout}
\usepackage{dsfont}
\usepackage{xcolor}
\usepackage{graphicx}
\usepackage{mathtools}
\mathtoolsset{showonlyrefs}
\usepackage{orcidlink}
\usepackage[authoryear,round]{natbib}

\newtheorem{theorem}{Theorem}[section]
\newtheorem{proposition}[theorem]{Proposition}
\newtheorem{lemma}[theorem]{Lemma}
\newtheorem{corollary}[theorem]{Corollary}
\newtheorem{remark}[theorem]{Remark}

\newcommand{\p}{\mathbb{P}}
\newcommand{\E}{\mathbb{E}}
\newcommand{\e}{\mathrm{e}}

\newcommand{\reals}{\mathbb{R}}
\newcommand{\ind}{\mathds{1}}

\newcommand{\Wq}{W^{(q)}}
\newcommand{\Wqprime}{W^{(q) \prime}}

\newcommand{\Hq}{H^{(q)}_{K,S}}
\newcommand{\Hqprime}{H^{(q) \prime}_{K,S}}

\newcommand{\F}{\mathcal{F}}
\newcommand{\diff}{\mathrm{d}}
\newcommand{\delx}{\partial_x}

\newcommand{\ettap}[1]{\frac{S}{q} + \frac{K}{q+K}\left(#1 + \frac{\mu - S}{q}\right)}
\newcommand{\ettam}[1]{\frac{S}{q} - \frac{K}{q+K}\left(#1 + \frac{\mu - S}{q}\right)}
\newcommand{\ettaz}{\frac{S}{q} + \frac{K}{q+K}\frac{\mu - S}{q}}

\renewcommand{\geq}{\geqslant}
\renewcommand{\leq}{\leqslant}

\allowdisplaybreaks

\begin{document}

\title[L{\o}kka-Zervos dichotomy for AC dividend strategies]{A new perspective on the L{\o}kka-Zervos dichotomy for absolutely continuous dividend strategies}

\author[]{Tommy Mastromonaco \orcidlink{0009-0000-5095-2374}}
\address{D\'epartement de math\'ematiques, Universit\'e du Qu\'ebec \`a Montr\'eal (UQAM), 201 av.\ Pr\'esident-Kennedy, Montr\'eal (Qu\'ebec) H2X 3Y7, Canada}
\email{mastromonaco.tommy@uqam.ca, fendrinacer11@gmail.com, renaud.jf@uqam.ca}
\email[]{simard.clarence@uqam.ca }

\author[]{Nacer Fendri}

\author[]{Jean-Fran\c{c}ois Renaud \orcidlink{0009-0002-9910-8523}}

\author[]{Clarence Simard \orcidlink{0000-0003-3860-3432}}
% \thanks{*Corresponding author: simard.clarence@uqam.ca .}

\date{\today}


\begin{abstract}
We revisit the optimization problem (and its dichotomous solution) analyzed in \cite{renaud-et-al_2023}. By choosing affine functions (of the surplus level) to bound the dividend rates, we are able to provide a more explicit derivation of the solution, avoiding the use of viscosity solutions. Moreover, with explicitly parameterized expressions, we are able to study analytical properties of the optimal thresholds and to refine the statement of the dichotomy. To reach these objectives, we add a penalty-at-ruin parameter in the performance function, allowing us to unify the expressions of the value functions of the two subproblems appearing in the dichotomous solution; as a by-product, the dichotomy can also be expressed in terms of this newly added parameter. Finally, we perform sensitivity analyses to illustrate that the range of values generated in an affine-bound framework is flexible and can span the range from a constant bound to the singular version of this problem.
\end{abstract}

\maketitle

\section{Introduction}

In this short paper, we revisit the optimization problem analyzed in \cite{renaud-et-al_2023}, i.e., the maximization of absolutely continuous dividends, in a Brownian risk model, with the possibility (not the obligation) of making capital injections. More precisely, in \cite{renaud-et-al_2023} absolutely continuous (AC) dividend strategies with payment rates bounded by a non-decreasing concave function of the capital surplus are considered. With such a general bound function, explicit expressions are unavailable for the value function and for the optimal strategy identified. In this work, we consider affine bound functions to offer not only a more explicit and transparent derivation of the solution, but also new results on the behaviour of the optimal thresholds with respect to the parameters of the bound function. Furthermore, to highlight the trade-off between the maximization of dividend payments and the possibility of declaring ruin, we add a penalty-at-ruin parameter to the performance function, in the spirit of \cite{shreve-lehoczky-gaver_1984}. By isolating this feature of the problem, we are better able to compare the performance functions of strategies without any injections and those with mandatory injections to keep the surplus from becoming negative.

\subsection{Related literature}

It is now well known, thanks to the pioneering work of \cite{lokka-zervos_2008} who considered the problem of maximizing (singular) dividends with the possibility (but not the obligation) of injecting capital, that the solution takes the form of a dichotomy: either capital is always injected whenever the firm faces bankruptcy, or no capital is ever injected. In other words, their solution to the joint optimization of capital injections and dividend payments is based on the solutions of two auxiliary/subproblems: the pure maximization of dividends, without the possibility of injecting capital, and the maximization of dividends with mandatory last-minute injections to prevent ruin. The pure dividend maximization problem often bears the name of Bruno de Finetti, see \cite{definetti_1957}, and there is now a wealth of publications on different versions of it; see, e.g., \cite{jeanblanc-shiryaev_1995,asmussen-taksar_1997,schmidli_2006,avram-palmowski-pistorius_2007}. The second family of optimization problems, sometimes called \textit{bailout optimal dividends problems}, incorporates ruin-preventing capital injections in the optimal dividends problem; see, e.g., \cite{shreve-lehoczky-gaver_1984, ST2002}.

Several papers have considered modifications of the problem studied in \cite{lokka-zervos_2008} and have also obtained dichotomous solutions; for example, in \cite{Zhu_2017}, the problem with AC dividend strategies with payment rates bounded by a constant, but with impulse capital injections, was considered. A recent trend in (pure) dividend optimization considers AC dividend payments with payment rates bounded by a function of the surplus. In \cite{renaud-simard_2021}, the bound function is an affine function while in \cite{locas-renaud_2024} a more general non-decreasing concave function is considered. Recall that this has been partly motivated by the objective of obtaining optimal strategies with the desirable property of yielding a stable flow of dividends. In this new framework, the problem of L\o kka \& Zervos was revisited by \cite{renaud-et-al_2023}. The dichotomous solution turned out to be much more difficult to obtain and sophisticated mathematical tools, such as viscosity solutions, were needed.

\subsection{Contributions}

In this paper, we provide a new perspective on the L\o kka-Zervos dichotomy by revisiting the AC version of the problem as studied in \cite{renaud-et-al_2023}. As alluded to above, we apply an affine bound function (of the current surplus) on the dividend rates and we add a penalty-at-bankruptcy component in the performance function. With these two modifications, we are able to:
\begin{itemize}
    \item provide an \textit{elementary} and more explicit derivation of the dichotomous solution, in contrast with the derivation in \cite{renaud-et-al_2023} which uses viscosity solution theory (see Section~\ref{sec:sol-main});
    \item prove new monotonicity and ordering properties of the optimal thresholds, which are instrumental in simplifying the derivation leading to the optimal dichotomy (see Proposition~\ref{prop:seuils-croiss} and Lemma~\ref{lem-delxJd});
    \item capitalize on the above more explicit results to offer a more precise formulation of the dichotomy, in particular proving that, for some set of parameters, capital injections are suboptimal regardless of the cost of capital injections and \textit{liquidation as quickly as possible} is optimal (see Section~\ref{subsec:result});
    \item disentangle the impact of declaring ruin from the value of dividend payments and capital injections, thanks to the addition of a penalty-at-ruin parameter in the performance function, and use it as a dual parameter of the cost-of-injection parameter (see Section~\ref{sec:sol-main});
    \item perform sensitivity analyses with respect to the two bound function parameters, which had not been done before for this problem, to illustrate various behaviours (see Section~\ref{sec:numerical}).
\end{itemize}

It must be pointed out that the apparent loss of generality with affine bound functions turns out to be negligible. Indeed, the generated by this AC problem with an affine bound function can be made as close as one wants to the value of the singular problem of L\o kka \& Zervos (it suffices to take the slope as large as necessary). This is illustrated in Section~\ref{sec:numerical}.

It is well known that the addition of a penalty-at-ruin component in the performance function does not offer any new mathematical challenge. However, it allows us to provide a unified representation of the value functions for the two subproblems; see Equations~\eqref{Jd-compact}-\eqref{Jc-compact} in Section~\ref{subsec:unified}. As a consequence, we are also able to formulate the dichotomy in terms of this parameter, instead of just the cost of injections as in all other studies on similar problems; see Corollary~\ref{cor:sol-main}.

\subsection{Formulation of the main optimization problem}\label{sec:framework}

Let us consider a filtered probability space $(\Omega, \F, \mathbb{F} = (\F_t)_{t\geq 0}, \p)$ where $\mathbb{F}$ supports a standard Brownian motion $B = (B_t)_{t\geq 0}$ and satisfies the usual conditions. We suppose that the (uncontrolled) surplus process evolves according to an arithmetic Brownian motion $X=(X_t)_{t\geq 0}$ given by
\begin{equation}
    X_t = X_0 + \mu t + \sigma B_t, \quad t\geq 0,
\end{equation}
where $\mu,\sigma>0$ are the model parameters. It is well known that, for a twice continuously differentiable function $f$, the infinitesimal generator of $X$ is given by
\begin{equation}
    \mathcal{A}f(x) = \mu f^\prime(x) + \frac{1}{2}\sigma^2f^{\prime \prime}(x) .
\end{equation}

As dividends can be deducted and capital injections can be made, a strategy is given by a pair of adapted processes $\pi = (\ell^\pi, G^\pi)$, where $\ell^\pi=(\ell^\pi_t)_{t\geq 0}$ is the dividend rate process, which is such that the amount of dividend payments made up to time $t$ is given by $L^\pi_t = \int_0^t \ell^\pi_s\diff s$, and where $G^\pi=(G^\pi_t)_{t\geq 0}$ is the capital injection process such that $G^\pi_t$ is the amount of capital injections made up to time $t$. The corresponding surplus process is then given by
\begin{equation}
    \diff X^\pi_t = (\mu-\ell^\pi_t)\diff t + \sigma \diff B_t + \diff G^\pi_t, \quad t\geq 0,
\end{equation}
and the corresponding ruin time by
\begin{equation}
    \tau^\pi = \inf\{t \geq 0 \mid X^\pi_t < 0\}.
\end{equation}
Note that in this problem, for a given $\pi$, the ruin time $\tau^\pi$ can be infinite with probability one or it can be finite with a positive probability.

Finally, as alluded to above, the set of admissible strategies is, for fixed $K,S \geq 0$, given by
\begin{align}
    \Pi = \{\pi = (\ell^\pi, G^\pi) &\mid 0 \leq \ell^\pi_t \leq KX^\pi_t + S \text{ for all $t\geq 0$} \\
    &\qquad \text{and $G^\pi$ is càdlàg  nondecreasing with $G^\pi_{0-} = 0$}\}.
\end{align}
On this set, we seek to maximize the following performance function: for $\pi \in \Pi$,
\begin{equation}\label{J-gen}
    J(x;\pi) = \E_x \left[\int_0^{\tau^\pi}\e^{-qt}\left(\ell^\pi_t \diff t - \beta\diff G^\pi_t\right) - P\e^{-q\tau^\pi} \right] , \quad x \geq 0 ,
\end{equation}
where $q>0$ is the discount rate, $\beta >1$ is the cost of capital injections, and $P\geq 0$ is the penalty-at-ruin parameter. We will refer to $K,S,q,\beta,P$ as the problem parameters. Then, the goal is to compute the value function
\begin{equation}
    V(x) = \sup_{\pi\in\Pi}J(x;\pi) , \quad x\geq 0 ,
\end{equation}
and find an optimal strategy (if it exists), i.e., a strategy $\pi^*\in\Pi $ such that $V(x) = J(x;\pi^*)$ for all $x\geq 0$.

\begin{remark}
Note that, in the above performance function, $\E_x$ stands for the expectation with respect to the probability measure $\p_x$ associated with the initial surplus $X^\pi_{0-} = x$.
\end{remark}

Roughly speaking, the solution to this optimization problem is given by the following dichotomy (see \cite{renaud-et-al_2023}): there exists a critical value $\beta^*$ such that
\begin{itemize}
\item if the cost of capital injections $\beta$ is larger than $\beta^*$, then an optimal strategy consists in paying dividends at the maximal rate, whenever the surplus exceeds an optimal threshold, and never injecting capital,
\item if the cost of capital injections $\beta$ is smaller than $\beta^*$, then an optimal strategy consists in paying dividends at the maximal rate, whenever the surplus exceeds another optimal threshold, and always injecting capital, only to keep the surplus process from going below zero.
\end{itemize}

A precise statement and a new proof of this dichotomy are provided in Section~\ref{sec:sol-main}.

The rest of the paper is organized as follows. Section~\ref{sec:sub-problems} contains known results about the two optimization subproblems mentioned earlier; it serves mainly to set the notation. The main contributions of this paper are in Section~\ref{sec:sol-main}, while in Section~\ref{sec:numerical} numerical analyses are presented to illustrate those results.

%%%%%%%%%%%%%%%%%%%%%%%%%%%%%%%%%%%%%%%%%
%%%%%%%%%%%%%%%%%%%%%%%%%%%%%%%%%%%%%%%%%
%%%%%%%%%%%%%%%%%%%%%%%%%%%%%%%%%%%%%%%%%
\section{Two optimization (sub)problems}\label{sec:sub-problems}

Now, we have to spend some time presenting the solutions to the two subproblems appearing in the dichotomous solution of the main optimization problem. Recall that they are known results already available in the literature.

Let us consider the following subsets of strategies:
\begin{equation}
    \Pi_d = \{\pi\in\Pi \mid G^\pi \equiv 0\} \quad \text{and} \quad \Pi_c = \{\pi\in\Pi \mid X^\pi_t \geq 0 \text{ for all } t\geq 0\} .
\end{equation}
The first subset contains strategies in which no capital injections are made, while the second subset contains those in which capital injections must be made (if needed) to keep the surplus process from getting ruined. Let us define the corresponding value functions by
\begin{equation}
    V_d(x) = \sup_{\pi\in\Pi_d} J(x;\pi) \quad \text{and} \quad V_c(x) = \sup_{\pi\in\Pi_c}J(x;\pi) ,
\end{equation}
in which, for each $\pi\in\Pi_d$, the performance function in~\eqref{J-gen} becomes
\begin{equation}
    J(x;\pi) = \E_x\left[\int_0^{\tau^\pi}\e^{-qt}\ell^\pi_t \diff t - P\e^{-q\tau^\pi} \right] ,
\end{equation}
while, for each $\pi\in\Pi_c$, it becomes
\begin{equation}
    J(x;\pi) = \E_x\left[\int_0^\infty \e^{-qt}\left(\ell^\pi_t \diff t - \beta\diff G^\pi_t\right)\right] ,
\end{equation}
since for such strategies we have $\tau^\pi = \infty$.

Note that, because $\Pi_d \subset \Pi$ and $\Pi_c \subset \Pi$, we have
\begin{equation}
    V(x) \geq \max\{V_d(x), V_c(x)\},\quad x\geq 0.
\end{equation}

Next, let us define \textit{special functions}, related to an arithmetic Brownian motion and an Ornstein-Uhlenbeck process, through Laplace transforms of first hitting times; the exact analytical expressions are given in Appendix~\ref{app:scalefunctions}. If $Y_t = x + \mu t + \sigma B_t$, then for $x\in [0,b]$
\begin{equation}
\E_x\left[\e^{-q\nu^b}\ind_{\nu^b < \nu^0}  \right] = \frac{\Wq(x)}{\Wq(b)} ,
\end{equation}
while, if $Y$ is such that $\diff Y_t = (\mu - (KY_t + S)) \diff t + \sigma \diff B_t$ with $Y_0 = x \geq b$, then
\begin{equation}
\E_x\left[\e^{-q\nu^b}\ind_{\nu^b<\infty} \right] = \frac{\Hq(x)}{\Hq(b)} .
\end{equation}
In each case, we used the notation $\nu^c = \inf\{t\geq 0 \mid Y_t = c\}$, for $c \in \reals$.

Here is the solution to the first subproblem, which is essentially lifted from \cite{renaud-simard_2021, rao_2023, fendri_2025}. Its proof is merely an exercise, so it is left to the reader.

\begin{theorem}\label{thm:sol-Jd}
    If
    \begin{equation}\label{cond-params-bd}
        \frac{q}{q+K}\frac{\Hq(0)}{\Hqprime(0)} + \ettaz + P > 0,
    \end{equation}
    then there exists a unique $b_d>0$ such that
    \begin{equation}\label{Vd-proba}
        V_d(x) = \E_x \left[ \int_0^{\tau_d} \mathrm{e}^{-qt} (KX^{b_d}_t +S) \ind_{X^{b_d}_t \geq b_d} \diff t - P\e^{-q\tau_d}\right]        
    \end{equation}
    with $\tau_d = \inf\{t\geq 0\mid X^{b_d}_t < 0\}$, where $X^{b_d}$ is the solution to
    \begin{equation}
        \diff X^{b_d}_t = \left(\mu - (K X^{b_d}_t + S) \ind_{X^{b_d}_t \geq b_d} \right) \diff t + \sigma \diff B_t .
    \end{equation}
    
    Otherwise, if the condition in Equation~\eqref{cond-params-bd} fails, then we have
    \begin{equation}\label{Vd-bd0-proba}
        V_d(x) = \E_x \left[ \int_0^{\tau_d} \mathrm{e}^{-qt} (KX'_t +S) \diff t - P\e^{-q\tau_d}\right]
    \end{equation}
    with $\tau_d = \inf\{t\geq 0\mid X'_t < 0\}$, where $X'$ is the solution to
    \begin{equation}
        \diff X'_t = \left(\mu - (K X'_t + S) \right) \diff t + \sigma \diff B_t .
    \end{equation}
\end{theorem}

It is known that the value function $V_d \colon [0,\infty) \to \reals$ is a concave function and that it is a solution to the following HJB equation:
\begin{equation}\label{HJB-d}
    (\mathcal{A}-q)f(x) + \sup_{0\leq\lambda\leq Kx+S}\left[\lambda(1-f'(x))\right] = 0, \quad x > 0,
\end{equation}
with boundary condition
\begin{equation}\label{bound-cond-d}
    f(0) = - P.
\end{equation}

Here is the solution to the second subproblem, taken from \cite{renaud-et-al_2023, mastromonaco_2025}.

\begin{theorem}\label{thm:sol-Jc}
    There exists a unique $b_c>0$ such that
    \begin{equation}\label{Vc-proba}
        V_c(x) = \E_x \left[\int_0^\infty \e^{-qt}\left((K Y^{b_c}_t + S)\ind_{Y^{b_c}_t\geq b_c} \diff t - \beta\diff G^{b_c}_t\right)\right],        
    \end{equation}
    where $(Y^{b_c},G^{b_c})$ is the solution to the Skorokhod problem
    \begin{align*}
        \diff Y^{b_c}_t &= \left(\mu - (K Y^{b_c}_t + S)\ind_{Y^{b_c}_t\geq b_c}\right)\diff t + \sigma \diff B_t + \diff G^{b_c}_t ,\\
            G^{b_c}_t &= \int_0^t \ind_{Y^{b_c}_s = 0} \diff G^{b_c}_s .
    \end{align*}
\end{theorem}

It is known that the value function $V_c \colon [0,\infty) \to \reals$ is a concave function and that it is a solution to the following HJB equation:
\begin{equation}\label{HJB-c}
    (\mathcal{A}-q)f(x) + \sup_{0\leq\lambda\leq Kx+S}\left[\lambda(1-f'(x))\right] = 0, \quad x > 0,
\end{equation}
with boundary condition
\begin{equation}\label{bound-cond-c}
    f'(0) = \beta.
\end{equation}

%%%%%%%%%%%%%%%%%%%%%%
%%%%%%%%%%%%%%%%%%%%%%
%%%%%%%%%%%%%%%%%%%%%%
\section{Solution to the main optimization problem}\label{sec:sol-main}

In this section, we present our more elementary derivation of the solution to the main optimization problem. First, in Section~\ref{subsec:unified}, we unify the analytical representations of the two performance functions related to the subproblems studied previously, this unification is instrumental in our approach. Then, in Section~\ref{subsec:J-prop}, we derive some analytical properties of these two functions that will be useful in Section~\ref{subsec:result}.

\subsection{A unified representation of the value functions}\label{subsec:unified}

For a given $b\geq 0$, let us consider the process $X^b$ defined by
\begin{equation}
    \diff X^b_t = \left(\mu - (KX^b_t + S)\ind_{X^b_t\geq b}\right)\diff t + \sigma \diff B_t, \quad t\geq 0 ,
\end{equation}
and the pair of processes $(Y^b, G^b)$ which is a solution to the following Skorokhod problem (see, e.g., Lemma~3.6.14 in \cite{karatzas-shreve_1991}):
\begin{align}
    \diff Y^b_t &= \left(\mu - (KY^b_t + S)\ind_{Y^b_t\geq b}\right)\diff t + \sigma \diff B_t + \diff G^b_t, \quad t\geq 0 ,\\
    G^b_t &= \int_0^t \ind_{Y^b_s = 0} \diff G^b_s .
\end{align}
Note that $Y^b_t \geq 0$ for all $t\geq 0$.

Define
\begin{equation}
    \ell^b_t = (KX^b_t + S)\ind_{X^b_t\geq b} \quad \text{and} \quad \underline \ell^b_t = (KY^b_t + S)\ind_{Y^b_t\geq b}.    
\end{equation}
In what follows, the (dividend-injection) strategies $\pi^b = (\ell^b, 0) \in \Pi_d$ and $\underline{\pi}^b = (\underline \ell^b, G^b) \in \Pi_c$ will be referred to as the \textit{pure-dividend mean-reverting strategy at level $b$} and the \textit{bailout-dividend mean-reverting strategy at level $b$}, respectively.

\begin{remark}
    Note that the optimal strategy in Theorem~\ref{thm:sol-Jd} is the pure-dividend mean-reverting strategy $\pi^{b_d} = (\ell^{b_d}, 0)$ while the optimal strategy in Theorem~\ref{thm:sol-Jc} is the bailout-dividend mean-reverting strategy $\underline \pi^{b_c} = (\underline \ell^{b_c}, G^{b_c})$. 
\end{remark}

\begin{remark}
    In \cite{renaud-et-al_2023}, a strategy $\underline{\pi}^b$ is called a \textit{reflected mean-reverting dividend strategy}. We have decided to use the word \textit{bailout} instead because it is more frequently used in the literature to designate ruin-preventing capital injections made each time the process reaches zero.
\end{remark}

Next, we slightly modify the notation of the performance function to emphasize the role played by the threshold when considering mean-reverting strategies. As mentioned in Section~\ref{sec:sub-problems}, for each $\pi^b = (\ell^b, 0)$ we have
\begin{equation}
    J(x;\pi^b) = J_d(x; b) := \E_x\left[\int_0^{\tau^b} \e^{-qt} \ell^b_t \diff t - P \e^{-q\tau^b} \right] ,
\end{equation}
where $\tau^b := \tau^{\pi^b}$, and for each $\underline \pi^b = (\underline \ell^b, G^b)$ we have
\begin{equation}
    J(x;\underline \pi^b) = J_c(x; b) := \E_x\left[\int_0^\infty \e^{-qt} \left( \underline \ell^b_t \diff t - \beta \diff G^b_t\right)\right] .
\end{equation}

Now, let us further the analysis of the performance functions $J_d(\cdot; b)$ and $J_c(\cdot; b)$. In this direction, let us define the expectations
\begin{equation}
\Phi(x;b) = \E_x\left[\e^{-q \tau^b}\right] \quad \text{and} \quad R(x;b) = \E_x\left[\int_0^{\tau^b}\e^{-qt}(KX^b_t + S)\ind_{X^b_t\geq b}\diff t\right], \quad x\geq 0 ,
\end{equation}
It is known (see \cite{mastromonaco_2025, renaud-et-al_2023}) that
\begin{equation}
    \Phi(x;b) =
    \begin{cases}
        \frac{\Wq(x-b)}{\Wq(-b)} +\Phi(b;b)\frac{\Wq(x)}{\Wq(b)}, &0\leq x< b, \\
        \hfil \Phi(b;b)\frac{\Hq(x)}{\Hq(b)}, &0\leq b \leq x,
    \end{cases}    
\end{equation}
where
\begin{equation}
    \Phi(b;b) = 
    \begin{cases}
        \hfil 1, &b = 0, \\
        -\frac{\Wqprime(0)/\Wq(-b)}{\frac{\Wqprime(b)}{\Wq(b)}-\frac{\Hqprime(b)}{\Hq(b)}}, &b>0,
    \end{cases} 
\end{equation}
and it is known (see \cite{renaud-simard_2021, rao_2023}) that
\begin{equation}
    R(x;b) =
    \begin{cases}
        \hfil R(b;b)\frac{\Wq(x)}{\Wq(b)}, &0\leq x< b, \\
        \ettap{x}+ \left( R(b;b) - \ettam{b} \right) \frac{\Hq(x)}{\Hq(b)}, &0\leq b \leq x,
    \end{cases}     
\end{equation}
where
\begin{equation}
    R(b;b) = 
    \begin{cases}
        \hfil 0, &b = 0, \\
        \frac{\frac{K}{q+K} - \left[\ettam{b}\right] \frac{\Hqprime(b)}{\Hq(b)}}{\frac{\Wqprime(b)}{\Wq(b)}-\frac{\Hqprime(b)}{\Hq(b)}}, &b>0.
    \end{cases}
\end{equation}

Since $X^b$ and $(Y^b, G^b)$ are Markovian processes, one can use the strong Markov property to obtain the following compact and unified expressions for $J_d(\cdot; b)$ and $J_c (\cdot; b)$: for $x\geq 0$,
\begin{equation}\label{Jd-compact}
J_d(x;b) = R(x;b) - P\Phi(x;b)
\end{equation}
and
\begin{equation}\label{Jc-compact}
    J_c(x;b) = R(x;b) + J_c(0;b)\Phi(x;b) .
\end{equation}

\begin{remark}
In both~\eqref{Jd-compact} and~\eqref{Jc-compact}, the first term is the present value of dividends until the surplus hits zero, while the second term is the residual value at ruin. In~\eqref{Jd-compact}, this second term is the discounted value of the penalty at ruin, while in~\eqref{Jc-compact} it is the discounted value (function) when the surplus restarts at zero. However, note that, in the latter case, capital injections keep the surplus afloat.
\end{remark}

Using the expressions for $R$ and $\Phi$, we can further write
\begin{equation}\label{Jd-byparts}
    J_d(x;b) =
    \begin{cases}
        \hfil -P\frac{\Wq(x-b)}{\Wq(-b)} + J_d(b;b)\frac{\Wq(x)}{\Wq(b)}, &0\leq x < b, \\
        \ettap{x} + \left(J_d(b;b) - \ettam{b} \right) \frac{\Hq(x)}{\Hq(b)}, &0\leq b \leq x,
    \end{cases}
\end{equation}
where
\begin{equation}
J_d(b;b) =
    \begin{cases}
        \hfil -P, &b = 0, \\
        \frac{\frac{K}{q+K} - \left[\ettam{b}\right]\frac{\Hqprime(b)}{\Hq(b)} + P \frac{\Wqprime(0)}{\Wq(-b)}}{\frac{\Wqprime(b)}{\Wq(b)}-\frac{\Hqprime(b)}{\Hq(b)}}, &b>0.
    \end{cases}
\end{equation}
Similarly,
\begin{equation}
    J_c(x;b) =
    \begin{cases}
        \hfil J_c(0;b)\frac{\Wq(x-b)}{\Wq(-b)} + J_c(b;b)\frac{\Wq(x)}{\Wq(b)}, &0\leq x < b, \\
        \ettap{x} + \left(J_c(b;b) - \ettam{b}\right)\frac{\Hq(x)}{\Hq(b)}, &0\leq b \leq x,
    \end{cases}
\end{equation}
where
\begin{equation}
    J_c(0;b) = \frac{\beta-\delx R(0;b)}{\delx\Phi(0;b)}, \;\; b\geq 0, \quad \text{and} \quad J_c(b;b) = \frac{\beta\Phi(b;b) + R(b;b)\frac{\Wqprime (-b)}{\Wq(-b)}}{\delx\Phi(0;b)}, \;\; b>0.
\end{equation}

\subsection{Analytical properties of the performance functions}\label{subsec:J-prop}

It is easy to see that both $J_d(\cdot;b)$ and $J_c(\cdot;b)$ are continuously differentiable on $(0,\infty)$, and twice continuously differentiable on $(0,\infty) \setminus \{b\}$. From Theorem~\ref{thm:sol-Jd} and Theorem~\ref{thm:sol-Jc}, we know there exist unique thresholds $b_d\geq 0$ and $b_c > 0$ such that, for all $x\geq 0$,
\[
V_d(x) = J_d(x;b_d) \quad \text{and} \quad V_c(x) = J_c(x;b_c) .
\]
The optimal thresholds $b_d$ and $b_c$ are characterized as the levels for which the (optimal) performance functions $J_d (\cdot;b_d)$ and $J_c (\cdot;b_c)$ are twice continuously differentiable on $(0,\infty)$. Analytically, they are the unique roots of $F$ and $G$, which are given explicitly in Equations~\eqref{def-F}-\eqref{def-G}; in other words, we have $F(b_d) = 0$ and $G(b_c)=0$.

Here is a new result in the analysis of the optimal thresholds. We write $b_d(P)$ and $b_c(\beta)$ to highlight the dependence on $P$ and $\beta$.
\begin{proposition}\label{prop:seuils-croiss}
    Suppose  $P^\prime$ satisfies equation \eqref{cond-params-bd} with $P<P^\prime$, then $b_d(P)<b_d(P^\prime)$. Suppose $1< \beta < \beta^\prime$, then $b_c(\beta)<b_c(\beta^\prime)$.
\end{proposition}
\begin{proof}
    First, if $P$ does not satisfy \eqref{cond-params-bd}, then $b_d(P) = 0 < b_d(P')$. Suppose now that $P$ does satisfy \eqref{cond-params-bd}. We write $F(b;P)$ (for $F$ defined in~\eqref{def-F}) to emphasize its dependence on $P$. Since $P \mapsto F(b;P)$ is strictly increasing, then $F(b_d(P);P^\prime) > F(b_d(P);P) = 0$, where the equality comes from the definition of $b_d(P)$. Because $b_d(P^\prime)$ is unique and $b \mapsto F(b;P)$ is continuous, then $F(b;P^\prime)$ crosses zero once with respect to $b$ at $b = b_d(P^\prime)$. Thus, since $F(0+)>0$ due to the fact that~\eqref{cond-params-bd} is satisfied, then $F(b;P^\prime) > 0$ if and only if $b\in (0,b_d(P^\prime))$. Consequently, $b_d(P)$ lies in this interval, so that $b_d(P) < b_d(P^\prime)$.
        
    Next, let us write $G(b;\beta)$ (for the function $G$ defined in~\eqref{def-G}) to highlight its dependence on the parameter $\beta$. Note that $G$ decreases with respect to $J_c(0;b)$ and the latter is strictly decreasing in $\beta$, so that $\beta \mapsto G(b;\beta)$ is strictly increasing. Thus, $G(b_c(\beta);\beta^\prime) > G(b_c(\beta);\beta) = 0$, where the equality comes from the definition of $b_c(\beta)$. Because $b_c(\beta^\prime)$ is unique and $b \mapsto G(b;\beta)$ is continuous, then $G(b;\beta^\prime)$ crosses zero once at $b = b_c(\beta^\prime)$. Thus, since $G(0+)>0$ (cf.\ \cite{mastromonaco_2025}), then $G(b;\beta^\prime) > 0$ if and only if $b\in (0,b_c(\beta^\prime))$. In consequence, $b_c(\beta)$ lies in this interval, so that $b_c(\beta) < b_c(\beta^\prime)$.
\end{proof}

Next, let us consider the following technical but important lemma.
\begin{lemma}\label{lem-delxJd}
    We have that:
    \begin{enumerate}
        \item $\delx J_d(0;b_d) \geq \delx J_d(0;b)$, for all $b\geq 0$;
        \item for $b\geq 0$, $\delx J_d(0;b) =\beta$ if and only if $J_c(0;b) = -P$;
        \item if $J_c(0;b_c) < -P$ (resp.\ $>$ or $=$), then $b_d < b_c$ (resp.\ $>$ or $=$).
    \end{enumerate}
\end{lemma}
\begin{proof}
Since $V_d = J_d(\cdot;b_d)$ and since, for any $b\geq 0$, $J_d(\cdot;b)$ is the performance function of a strategy $\pi^b \in \Pi_d$, then we have $J_d(x;b_d) \geq J_d(x;b)$ for all $x \geq 0$ and any $b\geq 0$. As $J_d(0;b)=-P$ for any $b \geq 0$, then we have $\delx J_d(0;b_d) \geq \delx J_d(0;b)$. This proves the first assertion.
    
Now, using the unified analytical representations of $J_d$ and $J_c$ (see~\eqref{Jd-compact} and~\eqref{Jc-compact}), and using the fact that $\delx J_c(0;b)=\beta$ for any $b \geq 0$, we have
\begin{equation}
\delx J_d(0;b) = \delx R(0;b) - P \delx\Phi(0;b)
\end{equation}
and
\begin{equation}
\delx J_c(0;b) = \delx R(0;b) + J_c(0;b) \delx\Phi(0;b) = \beta ,
\end{equation}
from which the second result follows.
    
For the last assertion, let us prove first the case corresponding to an equality, i.e., if $J_c(0;b_c) = -P$, then $b_c=b_d$. Suppose that $J_c(0;b_c) = -P$. Then, by the second assertion, we have $\delx J_d(0;b_c) =\beta$. Note from Theorem~\ref{thm:sol-Jd} and Theorem~\ref{thm:sol-Jc} that $V^\prime_d(b_d) = 1 = V^\prime_c(b_c)$, so that $b_d$ and $b_c$ uniquely satisfy $\delx J_d(b_d;b_d) = 1 = \delx J_c(b_c;b_c)$. Thus,
\begin{equation}
\delx R(b_d;b_d) - P\delx\Phi(b_d;b_d) = \delx R(b_c;b_c) - P\delx\Phi(b_c;b_c),  
\end{equation}
where we used the assumption that $J_c(0;b_c) = -P$. Clearly, $b_d = b_c$ solves this equation.
    
Suppose now that $J_c(0;b_c) < -P$ and set $P^\prime := -J_c(0;b_c) > P$. Then, by the result shown in the previous paragraph, we have $b_d(P^\prime)=b_c > 0$. Furthermore, we have $b_d(P^\prime) > b_d(P)$ by Proposition~\ref{prop:seuils-croiss}. Consequently, $b_c > b_d(P)$, that is, $b_d<b_c$. The same arguments hold with reverse inequalities in the case $J_c(0;b_c) > -P$.
\end{proof}

\subsection{Final steps toward the solution}\label{subsec:result}

Now, we have all the elements to complete the proof of the main result, i.e., Theorem~\ref{thm:sol-main} below.

As both $V_d$ and $V_c$ are solutions of the same HJB equation (see~\eqref{HJB-d} and~\eqref{HJB-c}), we will need to be able to compare the \textit{initial conditions} in~\eqref{bound-cond-d} and~\eqref{bound-cond-c}. This is the purpose of the following proposition.
\begin{proposition}\label{prop:equiv}
    We have $V_d'(0) > \beta$ if and only if $V_c(0) + P > 0$.
\end{proposition}
\begin{proof}
    Since $V_d = J_d(\cdot;b_d)$ and $V_c = J_c(\cdot;b_c)$, we can state the equivalence as: $\delx J_d(0;b_d) > \beta$ if and only if $J_c(0;b_c) + P > 0$. Recall that the inequality $J_c(0;b_c) + P > 0$ is equivalent to $\delx R(0;b_c) - P\delx\Phi(0;b_c) > \beta$. Consequently, by~\eqref{Jd-compact}, the equivalence now becomes: $\delx J_d(0;b_d) > \beta$ if and only if $\delx J_d(0;b_c) > \beta$.
    
    To prove the equivalence, we assume first that $\delx J_d(0;b_c) > \beta$. Taking $b$ equal to $b_c$ in statement~(1) of Lemma~\ref{lem-delxJd} yields $\delx J_d(0;b_d) > \beta$. Next, assume that $\delx J_d(0;b_d) > \beta$ and let us prove by contradiction that $\delx J_d(0;b_c) > \beta$. Suppose that $\delx J_d(0;b_c) \leq \beta$. If $\delx J_d(0;b_c) =\beta$, or equivalently if $J_c(0;b_c) = -P$ (by statement~(2) of Lemma~\ref{lem-delxJd}), then $b_c = b_d$ by statement (3) of Lemma~\ref{lem-delxJd}. Hence, $\delx J_d(0;b_d) = \delx J_d(0;b_c) = \beta$, which contradicts the assumption that $\delx J_d(0;b_d) > \beta$.
    Finally, if $\delx J_d(0;b_c) < \beta$, then $b_c < b_d$ by statement~(3) of Lemma~\ref{lem-delxJd}. By assumption, $\delx J_d(0;b_c) < \beta < \delx J_d(0;b_d)$, and $b\mapsto \delx J_d(0;b)$ is continuous, then there exists $b^*\in(b_c,b_d)$ such that $\delx J_d(0;b^*) = \beta$ by the intermediate value theorem. Thus, by statement~(2) of Lemma~\ref{lem-delxJd}, we have $J_c(0;b^*) = -P$. Using the same arguments as at the beginning of the proof, we see that $\delx J_d(0;b_c) < \beta$ and $J_c(0;b_c) < -P$ are equivalent. However, we have $-P = J_c(0;b^*) \leq J_c(0;b_c)$, which yields a contradiction.
\end{proof}

The inequalities in Proposition~\ref{prop:equiv} can be written more explicitly using the representations of $V_d$ and $V_c$ given in~\eqref{Vd-expli} and in~\eqref{Vc-expli}, respectively. Written differently, Proposition~\ref{prop:equiv} yields that
\begin{equation}
    \beta < -P\frac{\Wqprime (-b_d)}{\Wq(-b_d)} + \left(1+P\frac{\Wqprime (0)}{\Wq(-b_d)}\right)\frac{\Wqprime (0)}{\Wqprime (b_d)}
\end{equation}
if and only if
\begin{equation}
    P > \frac{-\Wq(-b_c)\left(\beta \Wqprime (b_c) - \Wqprime (0)\right)}{\Wqprime (-b_c) \Wqprime (b_c) - \Wqprime(0)^2} ,
\end{equation}
assuming that~\eqref{cond-params-bd} is satisfied. This equivalence will allow us to formulate the optimal solution of our problem using either $\beta$ or $P$ as a criterion for the dichotomy; see Theorem~\ref{thm:sol-main} and Corollary~\ref{cor:sol-main} below.

\begin{remark}
Of course, the criterion for optimality depends jointly on the parameters $\beta$ and $P$, but these two equivalent forms of the criterion provide interesting insight into the effect of each parameter on the optimal dichotomy.
\end{remark}

As a last intermediate step, here is the verification lemma of the main optimization problem.
\begin{lemma}\label{lem:verif}
    Let $f\in\mathcal{C}^2(\reals_+)$ be a concave solution to the HJB equation
    \begin{equation}\label{HJB}
        (\mathcal{A}-q)f(x) + \sup_{0\leq\lambda\leq Kx+S}\left[\lambda(1-f'(x))\right] = 0, \quad x > 0,
    \end{equation}
    with boundary condition
    \begin{equation}\label{bound-cond}
        \max\{-(f(0)+P), f'(0)-\beta\} = 0 .
    \end{equation}
Then, $f(x)\geq V(x)$ for all $x\geq 0$.
\end{lemma}
The proof of this verification lemma is standard (see, e.g., \cite{lokka-zervos_2008}), so it is left to the reader.

Now, using Lemma~\ref{lem:verif} with the properties of $V_d$ and $V_c$ and defining
\begin{equation}\label{beta-star}
    \beta^* := -P\frac{\Wqprime (-b_d)}{\Wq(-b_d)} + \left(1+P\frac{\Wqprime (0)}{\Wq(-b_d)}\right)\frac{\Wqprime (0)}{\Wqprime (b_d)},
\end{equation}
we can state and prove the main theorem.

\begin{theorem}\label{thm:sol-main}
If $\frac{q}{q+K}\frac{\Hq(0)}{\Hqprime(0)} + \ettaz + P > 0$, then $b_d>0$ and the solution of the main optimization problem is given by:
    \begin{itemize}
        \item if $\beta > \beta^*$, then an optimal dividend-injection strategy is given by $(\ell^{b_d}, 0)$, and thus $V = V_d$ with $b_d < b_c$;
        \item if $\beta < \beta^*$, then an optimal dividend-injection strategy is given by $(\underline \ell^{b_c}, G^{b_c})$, and thus $V = V_c$ with $b_c < b_d$;
        \item if $\beta = \beta^*$, then both aforementioned strategies are optimal, and thus $V=V_c=V_d$ with $b_c=b_d$.
    \end{itemize}
Otherwise, $b_d=0$ and an optimal strategy is given by the \textit{liquidation strategy} $(\ell^0, 0)$, and thus $V = V_d$.
\end{theorem}

\begin{proof}
First, recall that
\begin{equation}\label{V-VdVc}
V(x) \geq \max\{V_d(x), V_c(x)\}, \quad x\geq 0.
\end{equation}

Assume that $\frac{q}{q+K}\frac{\Hq(0)}{\Hqprime(0)} + \ettaz + P > 0$.
    \begin{itemize}
        \item In the first case, the condition on $\beta$ is equivalent to $V^\prime_d(0) < \beta$. We know from Section \ref{subsec:J-prop} that $V_d$ is a concave solution of \eqref{HJB} with boundary condition $V_d(0) = -P$. But since $V^\prime_d(0) < \beta$, it follows that $V_d$ also satisfies boundary condition \eqref{bound-cond}. Thus, by Lemma \ref{lem:verif}, $V_d \geq V$. By \eqref{V-VdVc}, $V\geq V_d$, thus, $V = V_d$. Finally, since $V^\prime_d(0) < \beta$ is equivalent to $V_c(0)<-P$ by Proposition \ref{prop:equiv}, then $b_d < b_c$ by statement (3) of Lemma \ref{lem-delxJd}.
        \item The condition on $\beta$ in the second case is equivalent to $V_c(0) > -P$. We know from Section \ref{subsec:J-prop} that $V_c$ is a concave solution of \eqref{HJB} with boundary condition $V^\prime_c(0) = \beta$. Coupled with $V_c(0) > -P$, it follows that boundary condition \eqref{bound-cond} is also satisfied by $V_c$. Hence, $V_c\geq V$ by Lemma \ref{lem:verif}, and $V\geq V_c$ by \eqref{V-VdVc}. In conclusion, $V = V_c$. Finally, $V_c(0) > -P$ implies that $b_c < b_d$ by statement (3) of Lemma \ref{lem-delxJd}.
        \item In the co-optimal case, the condition on $\beta$ is equivalent to $V^\prime_d(0) = \beta$ and $V_c(0) = -P$. Thus, both $V_d$ and $V_c$ satisfy boundary condition \eqref{bound-cond}, and as we know, they both are concave and satisfy \eqref{HJB}. Hence, $V_d\geq V$ and $V_c\geq V$. Additionally, we have $V\geq V_d$ and $V\geq V_c$ from \eqref{V-VdVc}, so we conclude that $V = V_d = V_c$. Finally, $b_d=b_c$ from statement (3) of Lemma \ref{lem-delxJd} since $V_c(0) = -P$. 
    \end{itemize}

Now, assume that $\frac{q}{q+K}\frac{\Hq(0)}{\Hqprime(0)} + \ettaz + P \leq 0$, in which case we have $b_d=0$ with $V_d$ given by~\eqref{Vd-bd0-expli} and such that
    \begin{equation}
        V^\prime_d(0) = \frac{K}{q+K} - \left(\ettaz + P\right)\frac{\Hqprime(0)}{\Hq(0)} \leq 1.
    \end{equation}
Hence, $V^\prime_d(0) < \beta$. The rest of the proof is as above, so we conclude that $V=V_d$ with $b_d=0$. 
\end{proof}

As mentioned earlier, the solution can also be stated from the point of view of parameter~$P$. Before doing so, let us define
\begin{equation}\label{P-star}
    P^* := \frac{-\Wq(-b_c)\left(\beta \Wqprime(b_c) - \Wqprime(0)\right)}{\Wqprime(-b_c) \Wqprime(b_c) - \Wqprime(0)^2} .
\end{equation}

\begin{corollary}\label{cor:sol-main}
If $\frac{q}{q+K}\frac{\Hq(0)}{\Hqprime(0)} + \ettaz + P > 0$, then $b_d>0$ and the solution of the main optimization problem is given by:
    \begin{itemize}
        \item if $P < P^*$, then an optimal strategy is given by $(\ell^{b_d}, 0)$, and thus $V = V_d$ with $b_d < b_c$;
        \item if $P > P^*$, then an optimal strategy is given by $(\underline \ell^{b_c}, G^{b_c})$, and thus $V = V_c$ with $b_c < b_d$;
        \item if $P = P^*$, then both aforementioned strategies are optimal, and thus $V=V_c=V_d$ with $b_c=b_d$.
    \end{itemize}
Otherwise, $b_d=0$ and an optimal strategy is given by $(\ell^0, 0)$, and thus $V = V_d$.
\end{corollary}

Recall that the dichotomy in this form has never appeared before in the literature. 

%%%%%%%%%%%%%%%%%%%%%%%%%%%%%%%%%%%%%%%%%%%%%
%%%%%%%%%%%%%%%%%%%%%%%%%%%%%%%%%%%%%%%%%%%%%
%%%%%%%%%%%%%%%%%%%%%%%%%%%%%%%%%%%%%%%%%%%%%
\section{Sensitivity analysis and discussion}\label{sec:numerical}

Now,  let us illustrate the theoretical results and discuss some interesting properties of the optimal solution. In Section~\ref{subsec:illust-sol}, we provide a sensitivity analysis of the value functions $V, V_d$ and $V_c$ with respect to the parameters $\beta$ and $P$, and then in Section~\ref{subsec:comp-prob} with respect to the parameters $K$ and $S$. Finally, in Section~\ref{subsec:num-K-S}, we analyze the behaviour of the optimal threshold (which is either $b_d$ or $b_c$) with respect to the parameters $K$ and $S$.

In what follows, we will say that the paramaters are \textit{favourable} if we have
\[
\frac{q}{q+K}\frac{\Hq(0)}{\Hqprime(0)} + \ettaz + P > 0 ,
\]
which has appeared in the statement of Theorem~\ref{thm:sol-Jd}.

\subsection{Sensitivity of the value function with respect to \texorpdfstring{$\beta$}{beta} and \texorpdfstring{$P$}{P}}\label{subsec:illust-sol}

First, let us illustrate the sensitivity of the value functions $V$, $V_d$ and $V_c$ with respect to $\beta$ and $P$.

\begin{figure}[b]
	\centering	
	\includegraphics[scale=0.7]{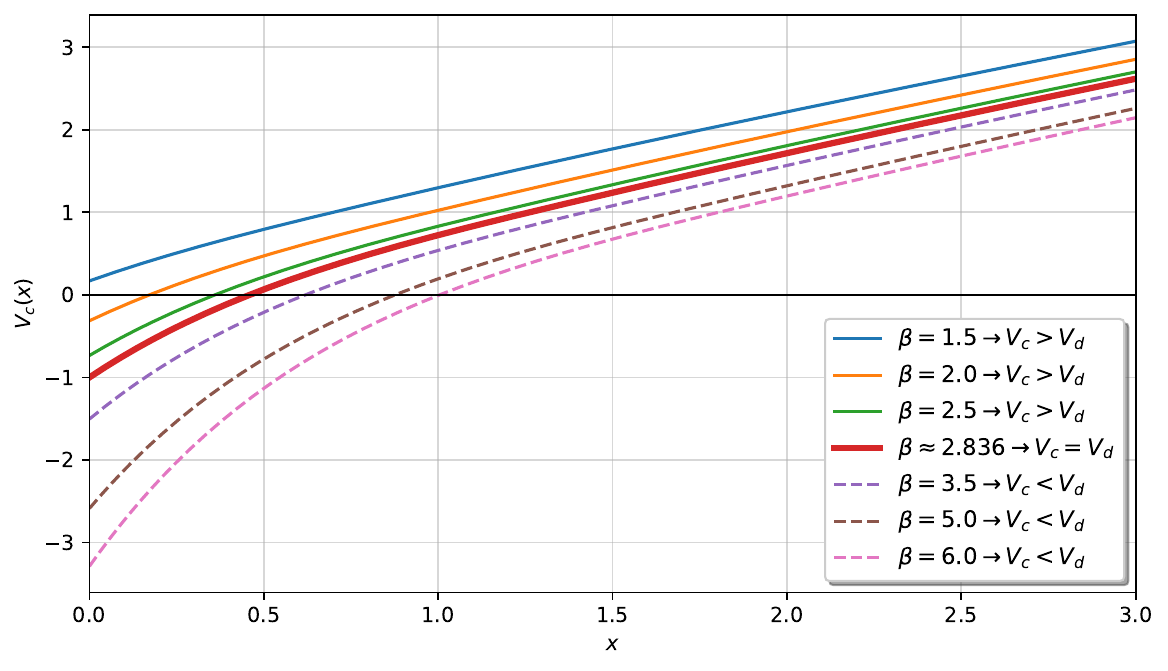}
	\caption{$V_c$ for different $\beta$ when parameters are \textit{favourable}}\label{fig:value-beta}
	\scriptsize
    Parameters: $\mu = 0.5, \; \sigma^2 = 1, \; q = 0.4, \; K = 1, \; S = 2, \; P = 1$.
\end{figure}

In Figure~\ref{fig:value-beta}, to illustrate Theorem~\ref{thm:sol-main}, we have drawn the graph of $V_c$ for different values of $\beta$. When $\beta$ is \textit{large}, then $V=V_d>V_c$ and we use dashed lines for $V_c$; when $\beta$ is \textit{small}, then $V=V_c>V_d$ and we use solid lines for $V_c$. More specifically, the set of parameters in Figure~\ref{fig:value-beta} is such that the switching value $\beta^*$ given by \eqref{beta-star} is approximately $2.8355$. Therefore, when $\beta > 2.8355$ we have $V=V_d > V_c$, and when $\beta < 2.8355$ we have $V=V_c>V_d$. Recall that the value function $V_d$ (given by the thick red line) does not depend on $\beta$.

\begin{figure}
	\centering	
	\includegraphics[scale=0.7]{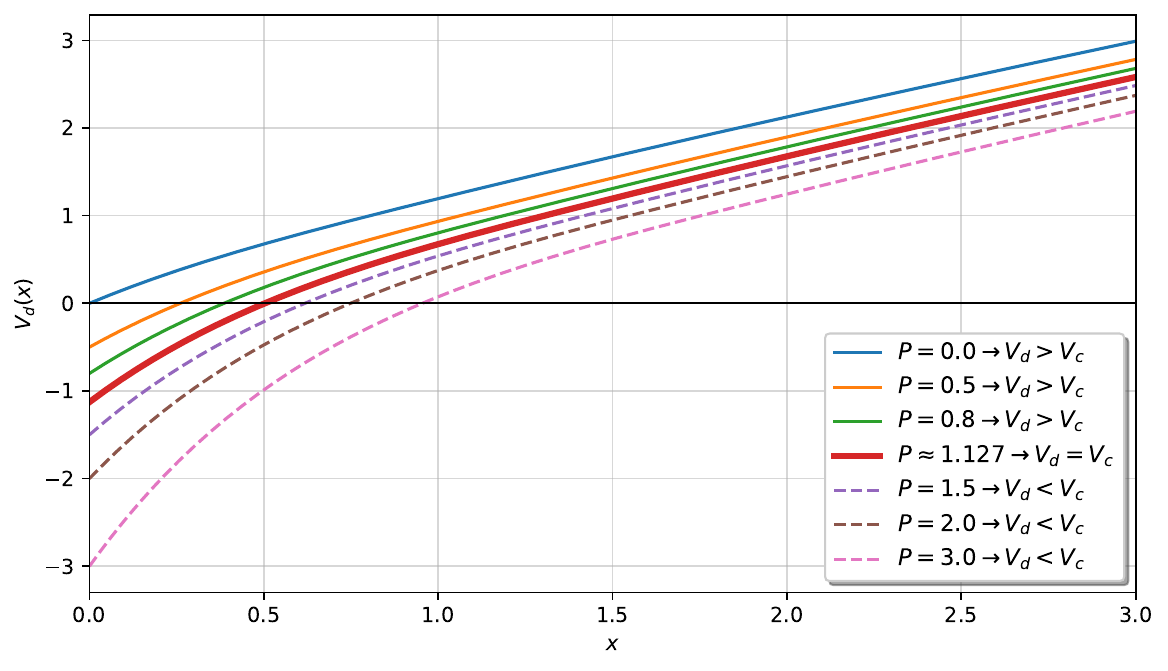}
	\caption{$V_d$ for different $P$ when the parameters are \textit{favourable}}\label{fig:value-P}
	\scriptsize
    Parameters: $\mu = 0.5, \; \sigma^2 = 1, \; q = 0.4, \; K = 1, \; S = 2, \; \beta = 3$.
\end{figure}

In Figure~\ref{fig:value-P}, to illustrate Corollary~\ref{cor:sol-main}, we have drawn the graph of $V_d$ for different values of $P$. When $P$ is \textit{large}, then $V=V_c>V_d$ and we use dashed lines for $V_d$; when $P$ is \textit{small}, then $V=V_d>V_c$ and we use solid lines for $V_d$. More specifically, the set of parameters in Figure~\ref{fig:value-P} is such that the switching value $P^*$ given by \eqref{P-star} is approximately $1.127$. Therefore, when $P > 1.127$ we have $V=V_c>V_d$, and when $P < 1.127$ we have $V=V_d>V_c$. Recall that the value function $V_c$ (given by the thick red line) does not depend on $P$.

\begin{figure}
	\centering	
	\includegraphics[scale=0.7]{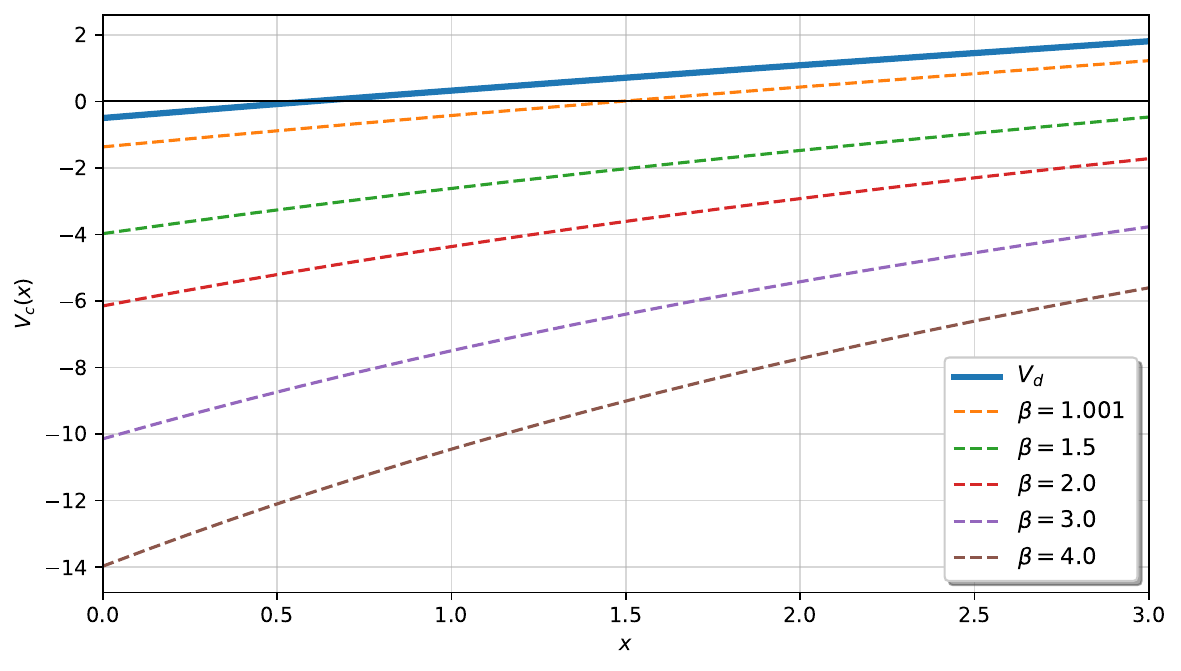}
	\caption{$V_c$ for different $\beta$ when the parameters are \textit{unfavourable}}\label{fig:value-no-dich}
	\scriptsize
    Parameters: $\mu = 0.5, \; \sigma^2 = 5, \; q = 0.8, \; K = 1, \; S = 2, \; P = 0.5$.
\end{figure}

Finally, in Figure~\ref{fig:value-no-dich}, we illustrate Theorem~\ref{thm:sol-main} in the situation where the parameters are \textit{unfavourable}, i.e., when it is optimal to \textit{liquidate as quickly as possible}. More precisely, for such parameters, we have $V(\cdot)=V_d(\cdot)=J_d(\cdot; 0)$, which means that the dividend rate is as high as possible, from the start until ruin. In Figure~\ref{fig:value-no-dich}, we have drawn the graph of $V_c$ for different values of $\beta$. In this case, no matter how small the value of $\beta$, we have $V=V_d>V_c$.

\subsection{Sensitivity of the value function with respect to \texorpdfstring{$K$}{K} and \texorpdfstring{$S$}{S}}\label{subsec:comp-prob}

\begin{figure}
	\centering	
	\includegraphics[scale=0.7]{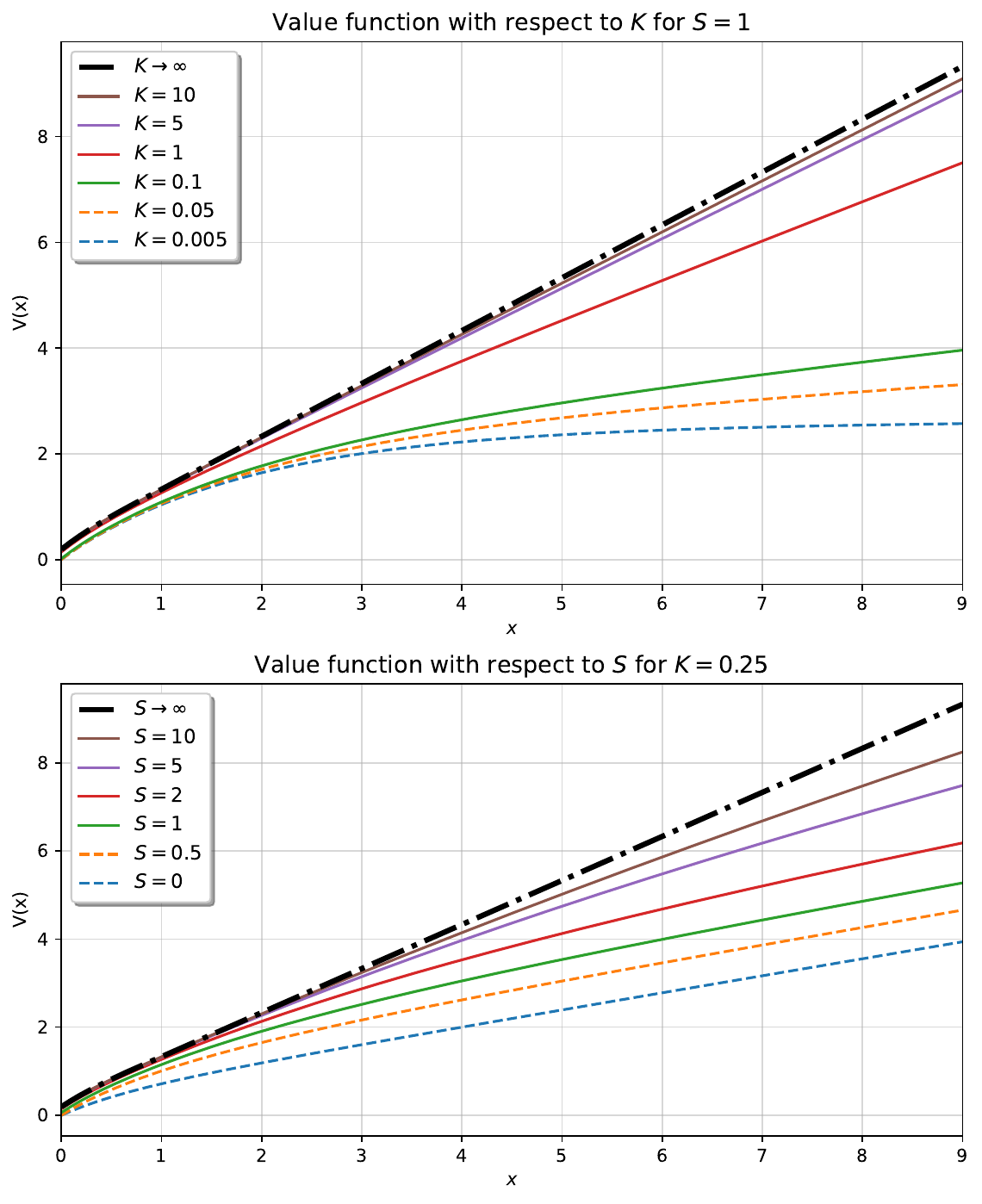}
	\caption{Value function with respect to $K$ or $S$ when $P=0$}\label{fig:value-K-S}
	\scriptsize
    A dashed line indicates that $V=V_d$ while a solid line indicates that $V=V_c$. The bold dashed-dotted lines represent the value function $V_c$ of the limit problem; that is, function $h = V$ in \cite{lokka-zervos_2008}. Parameters: $\mu = 0.5, \; \sigma^2 = 1, \; q = 0.4 , \; P=0$.
\end{figure}

In Figure~\ref{fig:value-K-S}, we illustrate the sensitivity of the value function $V$ with respect to the parameters $K$ and $S$, which characterize the set of admissible dividend payout rates. Intuitively, if $K$ and/or $S$ are large, a mean-reverting dividend strategy and its associated controlled process should be \textit{close} to a reflective barrier strategy and a reflected Brownian motion with drift, respectively, as studied in \cite{lokka-zervos_2008} (when $P=0$). For comparison purposes, in what follows we take $P=0$ and we use the notation $h$ for the value function of the singular control problem studied in \cite{lokka-zervos_2008}. Despite the fact that $V<h$, we want to illustrate the difference.

In Figure~\ref{fig:value-K-S}, the set of parameters is such that, in the singular control problem of L\o kka \& Zervos \cite{lokka-zervos_2008}, it is optimal to always inject when the process reaches zero. In our absolutely continuous control problem, we have $V=V_d>V_c$ for \textit{small} values of $K$ and/or $S$ (as indicated by the dashed lines), and we have $V=V_c>V_d$ for \textit{larger} values of $K$ and/or $S$. A first observation is that our value function $V$ gets closer to $h$ rapidly when either $K$ or $S$ becomes large, even if the other is small. This means that shareholders can almost reproduce the value generated from the optimal singular solution whilst enjoying the \textit{desirable properties} of a mean-reverting dividend strategy with an affine bound on the payout rates.  

\subsection{Sensitivity of the optimal threshold with respect to \texorpdfstring{$K$}{K} and \texorpdfstring{$S$}{S}}\label{subsec:num-K-S}

\begin{figure}
	\centering	
	\includegraphics[scale=0.7]{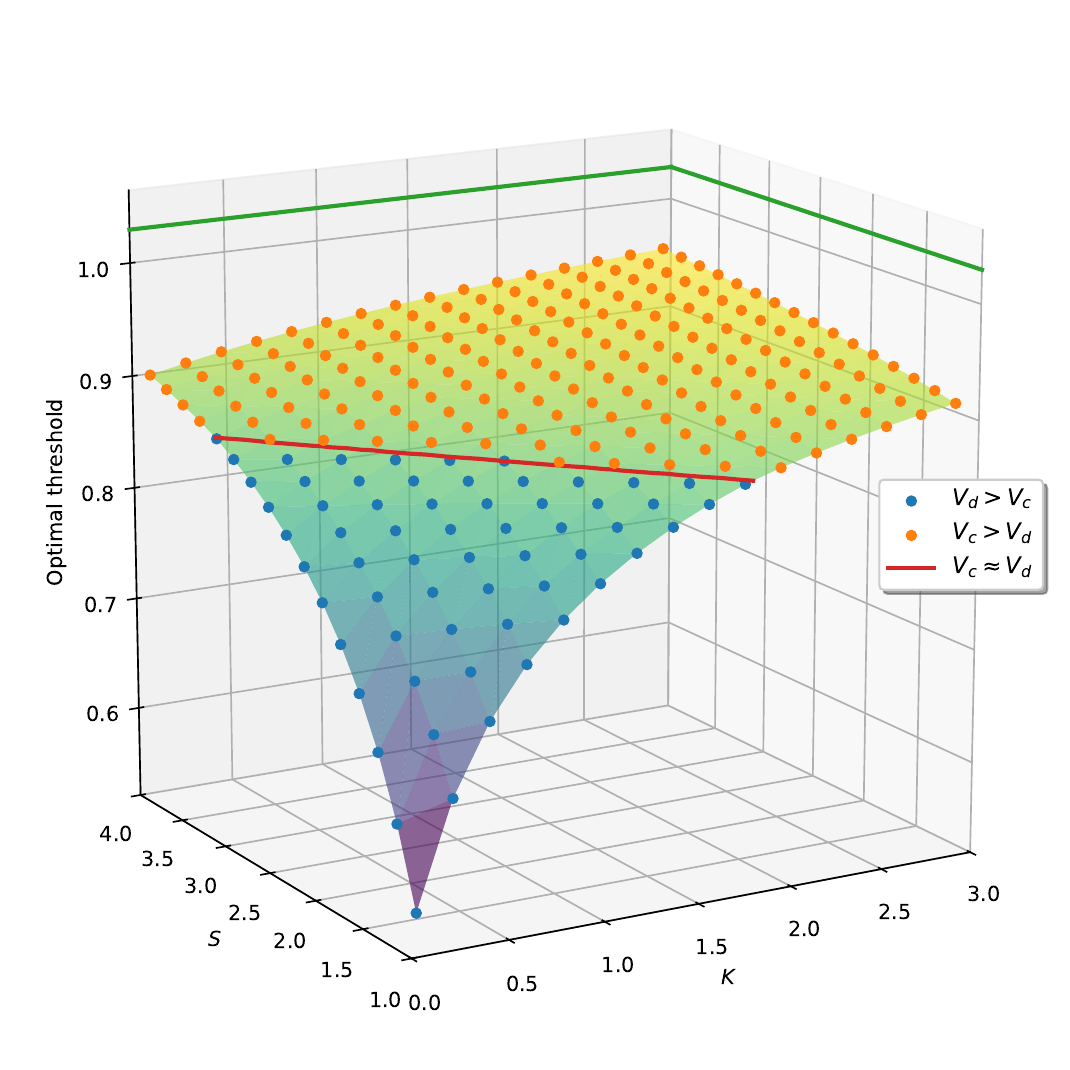}
	\caption{The optimal threshold with respect to $K$ and $S$ (and when $P=0$)}\label{fig:barrier-K-S}
	\scriptsize
    The parameter $\beta$ is chosen so that for $K=1$ and $S=2$, we have $V_c = V_d$. The red line numerically approximates the values of $K,S$ for which $V_c=V_d$. The green line corresponds to the optimal barrier of the singular problem in \cite{lokka-zervos_2008}. Parameters: $\mu = 0.5, \; \sigma^2 = 1, \; q = 0.4, \; P=0, \; \beta = 1.6633$.
\end{figure}

Figure~\ref{fig:barrier-K-S} illustrates the optimal threshold with respect to (small values of) $K$ and $S$. Blue dots indicate that a \textit{pure-dividend mean-reverting strategy} is optimal, while orange dots indicate that a \textit{bailout-dividend mean-reverting strategy} is optimal. The plot was first constructed by selecting $K=1$ and $S=2$, and then finding $\beta$ such that $V_c = V_d$. Afterward, we considered values of $K,S$ around these initial  values.

Figure~\ref{fig:barrier-K-S} suggests that $V=V_c$ for sufficiently large $K,S$, as mentioned above, and that the optimal threshold is increasing with respect to both $K$ and $S$. Intuitively, for large values of $K,S$, the optimal controlled process is \textit{strongly refracted} at the threshold, almost like a reflection; in other words, the process is kind of constrained to stay between zero and the threshold. In this case, it is better to have a higher threshold to stay away from zero, so as to avoid ruin or avoid making too many capital injections. Since the process is pushed more strongly towards zero and dividend payments are more rewarded, it is preferable to inject capital in order to avoid ruin and keep paying dividends.

Interestingly, there seems to be a somewhat linear boundary with respect to $K,S$ for the determination of the optimal strategy in the dichotomy, as shown by the red line in Figure~\ref{fig:barrier-K-S}. It indicates that there exist infinitely many $K,S$ such that $V_c=V_d$.

\section*{Funding and competing interests}

This work was supported by a CANSSI Qu\'ebec Recruitment Scholarship (CQRS), and by a Graduate Research Scholarship (BESRCD-612070-2026) and Discovery Grants (RGPIN-2025-05758 and RGPIN-2020-06619) from the Natural Sciences and Engineering Research Council of Canada (NSERC).

The authors have no competing interests to declare that are relevant to the content of this article.

All authors contributed equally to this work.
%%%%%%%%%%%%%%%%%%%%%%%%%%%%%%%%%%%%%%%%%%%%%
%%%%%%%%%%%%%%%%%%%%%%%%%%%%%%%%%%%%%%%%%%%%%
\bibliographystyle{chicago}
\bibliography{references_de_finetti_Tommy}
%%%%%%%%%%%%%%%%%%%%%%%%%%%%%%%%%%%%%%%%%%%%%%
%%%%%%%%%%%%%%%%%%%%%%%%%%%%%%%%%%%%%%%%%%%%%%%
\appendix

\counterwithin{equation}{section}   

\section{Special functions}\label{app:scalefunctions}

Here are analytical expressions for the functions introduced in Section~\ref{sec:sub-problems}. First, we have
\begin{equation}
    \Wq(x)= \frac{2}{\Delta} \mathrm{e}^{-\frac{\mu}{\sigma^2} x} \sinh \left(\frac{\Delta}{\sigma^2} x \right), \quad x\in\reals,
\end{equation}
where $\Delta = \sqrt{\mu^2 + 2\sigma^2 q}$. Also, we have
\begin{equation}
    \Hq(x) = \mathrm{e}^{\frac{K}{2\sigma^2} \left(x-\frac{\mu-S}{K}\right)^2} \mathrm{D}_{-\frac qK} \left(\left(x-\frac{\mu-S}{K} \right) \frac{\sqrt{2K}}{\sigma} \right), \quad x\geq 0,
\end{equation}
where
\begin{equation}
    \mathrm{D}_{-\lambda}(x) = \frac{1}{\Gamma(\lambda)} \mathrm{e}^{-x^2/4} \int_0^\infty t^{\lambda-1} \mathrm{e}^{-xt-t^2/2} \mathrm{d}t 
\end{equation}
is the parabolic cylinder function.

Finally, let us define, for $b \geq 0$,

\begin{equation}\label{def-F}
    F(b) = \frac{q}{q+K} \frac{\Hq(b)}{\Hqprime(b)} + \ettap{b} - \frac{\Wq(b)}{\Wqprime(b)} \left(1 + P \frac{\Wqprime(0)}{\Wq(-b)} \right)
\end{equation}
and
\begin{equation}\label{def-G}
    G(b) = \frac{q}{q+K}\frac{\Hq(b)}{\Hqprime(b)} + \ettap{b} - \frac{\Wq(b)}{\Wqprime (b)}\left(1 - g(b)\frac{\Wqprime (0)}{\Wq(-b)}\right),
\end{equation}
where
\begin{equation}
    g(b) := \frac{\beta\left(\frac{\Wqprime(b)}{\Wq(b)}-\frac{\Hqprime(b)}{\Hq(b)}\right) - \left(\frac{K}{q+K}-\left[\ettam{b}\right]\frac{\Hqprime(b)}{\Hq(b)}\right)\frac{\Wqprime(0)}{\Wq(b)}}{\frac{\Wq(-b)}{\Wqprime(-b)}\left(\frac{\Wqprime(b)}{\Wq(b)}-\frac{\Hqprime(b)}{\Hq(b)}\right) - \frac{\Wqprime(0)^2}{\Wq(b)\Wq(-b)}}.
\end{equation}

\section{Explicit expressions for the value functions of the subproblems}\label{app:Vd}

We also provide an explicit expression for the value function $V_d$ of Theorem~\ref{thm:sol-Jd}, as derived in \cite{renaud-simard_2021, rao_2023, fendri_2025}. If
    \begin{equation}
        \frac{q}{q+K}\frac{\Hq(0)}{\Hqprime(0)} + \ettaz + P > 0,
    \end{equation}
then
   \begin{equation}\label{Vd-expli}
        V_d(x) =
        \begin{cases}
            -P\frac{\Wq(x-b_d)}{\Wq(-b_d)} + \left(1 + P\frac{\Wqprime (0)}{\Wq(-b_d)}\right)\frac{\Wq(x)}{\Wqprime (b_d)}, &0\leq x< b_d, \\
            \hfil \ettap{x} + \frac{q}{q+K}\frac{\Hq(x)}{\Hqprime (b_d)}, &0 < b_d \leq x.
        \end{cases}
    \end{equation}
Otherwise, we have
    \begin{equation}\label{Vd-bd0-expli}
        V_d(x) = \ettap{x} - \left(\ettaz + P\right)\frac{\Hq(x)}{\Hq(0)}, \quad x\geq 0.
    \end{equation} 

Let us also provide an explicit expression for the value function $V_c$ of Theorem~\ref{thm:sol-Jc}, as derived in \cite{mastromonaco_2025, renaud-et-al_2023}. We have
   \begin{equation}\label{Vc-expli}
        V_c(x) =
        \begin{cases}
            \frac{\Wq(x)\left(\Wqprime (-b_c)-\beta \Wqprime (0)\right) + \Wq(x-b_c)\left(\beta \Wqprime (b_c) - \Wqprime (0)\right)}{\Wqprime (-b_c) \Wqprime (b_c) - \Wqprime (0)^2}, &0\leq x< b_c, \\
            \hfil \ettap{x} + \frac{q}{q+K}\frac{\Hq(x)}{\Hqprime (b_c)}, &0 < b_c \leq x.
        \end{cases}
    \end{equation}

\end{document}